\documentclass[12pt]{amsart}
\usepackage{amsfonts}

\usepackage[letterpaper]{geometry}
\usepackage{amsmath}
\usepackage{mathtools}
\usepackage{commath}

\usepackage{amssymb}
\def\R{\mathbb{R}}

\def\N{\mathbb{N}}

\newtheorem{thm}{\bf Theorem}[section]
\newtheorem{lemma}[thm]{\bf Lemma}
\newtheorem{prop}[thm]{\bf Proposition}

\theoremstyle{definition}
\newtheorem*{defi}{\bf Definition}

\newtheorem*{example}{\bf Example}

\begin{document}

\title[On the structures of subset sums in higher dimension] {On the structures of subset sums in higher dimension}

\author[]{Norbert Hegyv\'ari}
\address{Norbert Hegyv\'{a}ri, ELTE TTK,
E\"otv\"os University, Institute of Mathematics, H-1117
P\'{a}zm\'{a}ny st. 1/c, Budapest, Hungary and associate member of the Alfr\'ed R\'enyi Institute of Mathematics, H-1364 Budapest, P.O.Box 127.}
 \email{hegyvari@renyi.hu}

 \author[]{M\'at\'e P\'alfy}
\address{M\'at\'e P\'alfy, ELTE TTK,
E\"otv\"os University, Institute of Mathematics, H-1117
P\'{a}zm\'{a}ny st. 1/c, Budapest, Hungary}
 \email{palfymateandras@gmail.com}

\author[]{Erfei Yue}
\address{Erfei Yue, ELTE TTK,
E\"otv\"os University, Institute of Mathematics, H-1117
P\'{a}zm\'{a}ny st. 1/c, Budapest, Hungary}
 \email{yef9262@mail.bnu.edu.cn}

\large

\maketitle

\begin{abstract}
A given subset $A$ of natural numbers is said to be complete if every element of $\N$ is the sum of distinct terms taken from $A$. This topic is strongly connected to the knapsack problem which is known to be NP complete. The main goal of the paper is to study the structure of subset sums in a higher dimension. We show 'dense' sets and generalized arithmetic progrssions in subset sums of certain sets.

AMS 2010 Primary 11B30, 11B39, Secondary 11B75

Keywords: Structure of subset sums, thin bases, Szemer\'edi-Vu, Chen theorem on complete sequences, generalized arithmetic progression, matching problems
\end{abstract}

\section{Introduction}
The classical result of Lagrange's theorem states that the squares form a base of order 4, i.e. every integer can be written as the sum of four squares. 

Let $\N$ denote the set of positive integers, i.e. let $\N=\{1,2,\dots\}$. The set $A\subseteq \N$ is a basis, of order $h$ if every sufficiently large integer can be represented as the sum of exactly  $h$ not necessarily distinct elements of $A$. An easy counting argument shows that if $A$ is a basis of order $h$ then $\liminf_{n\to \infty}\frac{A(n)}{n^{1/h}}>0$.
A basis $A$ of order $h$ is said to be a thin basis if $\limsup_{n\to \infty}\frac{A(n)}{n^{1/h}}<\infty$. Let $S=\{1,4,9,\dots\}$ denote the squares. Wirsing showed in [W] that there exists an 'almost' thin basis of squares, proving that there exists an $S'\subseteq S$ such that $|S'\cap [1,x]|\ll (x\log x)^{1/4}$ and $S'$ is a basis of order four. ($f\ll g$ is the Vinogradov notation; it means that there is an absolute constant $C$ such that $|f|\leq C|g|$. $q\gg f$ means $f\ll g$). Spencer also gave a quick proof of this statement (see [S]).
For thin basis see Nathanson's paper [N] and the references therein to other authors.

For any $X\subseteq \N^k$ let
\begin{equation}\label{1}
FS(X):=\{\sum_{i=1}^\infty\varepsilon_ix_i: \ x_i\in X, \ \varepsilon_i \in \{0,1\}, \ \sum_{i=1}^\infty\varepsilon_i<\infty\}\end{equation}
Erd\H os called a sequence $A\subseteq \N$ {\it complete} if  every sufficiently large integer belongs to $FS(A)$. 

If $A\subseteq \N^k$ there are several definition of completeness (see [GFH, H1, H2]). We say that a set $X$ is {\it complete respect to the region} $R\subseteq \N^k$ if for every $z\in R$, $z\in FS(X)$ also holds. 

One of the main notion in our paper is the concept of a thin subset sum. It is known that every positive integer can be written in base 2, in other words, the set $\{2^n\}^\infty_{n=0}$ is a complete set. On the other hand if a set $A=\{a_1<a_2<\dots\}$ is a complete set  and $a_k\leq N< a_{k+1}$  then $k\geq \log_2N+c_0$ since there are $2^k$ many subset sums and $FS(A)$ covers the interval $[T,N]$ with some $T>0$.

Similarly in a higher dimension, the necessary condition that the subset sums of a subset $X\subseteq \N^k$ represent all far points of $\N^k$ should be the condition $X(N)>k\log_2N+t_X$ for some $t_X$, i.e. $X$ is complete respect to the region $R=\{x=(x_1,x_2,\dots,x_k):x_i\geq r_i\}$, $r_i\in \N$, $i=1,2,\dots, k$. In the case of $k=1$, the first named author gave a sufficient condition of a "thin" complete subsequence, (i.e. the existence of $B\subseteq A$ with $B(N)=\log_2N+t_B$ and $B$ is complete). This is analogous to the problem of thin basis (see [H2]).

\medskip

\begin{defi}
Let $A=\{a_1<a_2<\dots< a_n<\dots\}$ be an infinite sequence of integers. We say that $A$ is {\it weakly thin} if $\limsup_{n\to \infty }\frac{\log a_n}{\log n}=\infty$, or equivalently $A(n):= \sum_{a_i\leq n}1=n^{g(n)}$, where $A(n)$ is the counting function of $A$ and $\liminf_{n\to \infty }g(n)=0$. A set $B\subseteq \N$ is said to be {\it thick} if it is not weakly thin.
\end{defi}
We define the notion of a thin set in a higher dimension:
\begin{defi}
Let $X\subseteq \N^k$. 
$X$ is said to be {\it thin complete set} respect to $R$ if  $X(N)\leq log_2R(N)+t_X$ for some $t_X$ and $FS(X)\supseteq R$. In higher dimension $R(N)$  denotes the number of lattice points of $R\cap [1,N]^k$. If $R$ is a finite set we use $|R|$ for it. 
\end{defi}
Let us note that instead of the cube $[1,N]^k$, we can also use a sphere $B_N$ of dimension $k$ with radius $N$. The two counting functions differ only in one constant when considering (weakly) thin sequences.

It is not at all clear whether there exists 'thin' subset sums of any region of $\N^k$. 

It will turn out that if none of the tangents of $R$ are parallel to the axes, then there exists such a 'thin' subset sums. More precisely, the  region $R$ can be covered by a set of thin subset sums. (See the precise formulation in Proposition 1.1 and Theorem 1.2). 

\smallskip

In section 3 we study the structure of reducible subset sums of splittable sequences in $\N^2$. A sequence $A=\{a_i\}$ is splittable if any element $a_i$ can be split into sum of earlier elements, if $i$ is big enough. Par excellence $\{2^k\}_{k=0}$, the Fibonacci sequence e.t.c are splittable.

\smallskip

In section 4 we are going to study the structure of the subset sum of  $\{a_m\}\times\{b_k\}_{m,k\in \N}$ where $\{a_m\}$ is denser. We show that under some condition $FS(\{a_m\}\times\{b_k\})$ contains generalized arithmetic progressions and 'dense' rectangles. 

The difficulty of finding $FS(A\times B)$ structures is to match equal parts between $A$ and $B$ elements and to show the desired structures.  In the present work we can avoid this tool used in [BHP] and [GFH] (the so-called "Matching Principle").

\bigskip

\subsection{Thin subset sums} Denote by $f_i$ the $i^{th}$ axis of $\N^k$, i.e. let $f_i=\{(0,\dots,0, \alpha_i, 0,\dots 0):\alpha_i\in \N\}$ (recall that $\N$ denotes the positive integers).
In section 2 we prove
\begin{prop}\label{T1}
Let $z\in \N^k$, $k\geq 2$ and let $R=z+\N^k$. Suppose that $A$ is complete with respect to $R$. Then all projections of $A$ onto $f_i$ are thick.
\end{prop}
Note that in the proposition above the boundary of the domain $R$ contains half-lines parallel to the axis. In the next theorem, we will show that this was the only reason why a region does not have a thin complete subset.
\begin{thm}\label{T2}
Suppose~$v_1,\ldots,v_k\in\mathbb{N}^k$ are integer vectors, and they are linearly independent in~$\mathbb{R}^n$.
In addition, suppose none of~$v_i$'s is parallel to a coordinate axis. 
Let~$U$ be the cone generated by them, that is,
\begin{equation*}
U:=\{a_1v_1+\cdots+a_kv_k \in\mathbb{R}^k\mid  a_i\in \R_{\geq 0}, \ i=1,2\dots,k\},
\end{equation*}
and~$V=U\cap\mathbb{N}^n$ be the collection of all integer points in it. 
Let~$X=S\cup(X_1\cup\cdots\cup X_k)$, where~$S$ is the set of all integer points in the simplex generated by~$kv_1,\ldots,kv_k$, and~$X_i=\{v_i,2v_i,4v_i,8v_i,\ldots\}$. 
Then~$X$ is thin, and~$V=\mathrm{FS}(X)$.
\end{thm}

\section{Proof of Proposition \ref{T1} and Theorem \ref{T2}}

\begin{proof}[Proof of Proposition \ref{T1}] Suppose that, contrary to the assumption, there is $i$ for which the projection $\psi$ of $A$ onto $f_i$ is weakly thin (wlog assume $i=1$). Denote this projection $\psi(A)$ by $A_1\subseteq \N$. Let $z_1$ and $z_2$ be the first and second coordinates of $z$ respectively. Since $A_1$ is weakly thin we have an $N\in \N$, $N>2z_1$ for which $A_1(N)<N^{g(N)}$, where $g(N)<\frac{1}{2(z_2+1)}$. It implies that there is integer $z_1<x_1<N$ for which for any representation $x_1=\sum^t_{j=1}a_{1,j}\in tA_1\cap [1,N]$,  $t>z_2$ always holds. (Note that the numbers $a_{i_j}$ are not necessarily different). Indeed 
$$
|z_2A_1\cap [1,N[|=\big|\big\{\sum^r_{j=1}a_{1,j}\in rA_1: r\leq z_2\big\}\big|=\sum_{i\leq z_2}{A_1(N)+i-1\choose i}\leq 
$$
$$
\leq z_2(N^{g(N)}+z_2)^{z_2}<N-z_1.
$$
We claim that the point $(x_1,z_2+1, \dots)\in R$ does not belong to $FS(A)$. Let us assume the contrary.
$$
(x_1,z_2+1, \dots)=\sum^t_{j=1}(a_{1,j},a_{2,j},\dots).
$$
Recall that $t>z_2$ hence $\sum^t_{j=1}a_{2,j}>z_2+1$, since for all $j$, $a_{2,j}\geq 1$. This leads us to a contradiction. 

\end{proof}

We use geometric arguments to prove Theorem~\ref{T2}. There is an extensive literature on the topic of covering simplexes with homothetic copies. Here we are dealing with simple reasoning.

Suppose~$v_1,\ldots,v_k\in\mathbb{R}^k$ are vectors in general position. 
Let~$S$ be the simplex generated by them, that is, 
$$
S:=\{a_1v_1+\cdots+a_kv_k\mid a_1+\cdots+a_k\leq 1,a_1,\ldots,a_k\geq 0\}, 
$$
and
$$
F:=\{a_1v_1+\cdots+a_kv_k\mid a_1+\cdots+a_k=1,a_1,\ldots,a_k\geq 0\}
$$
be a face of $S$. For a fixed~$\lambda>0$ and each~$i\in[k]$, let~$S_i=\lambda v_i+S$ be the translation of~$S$ along~$v_i$,
and~$F_i=S_i\cap F$. Using some geometric observation, we can prove that if~$\lambda$ is small enough then there is no 'gap' between these~$S_i$'s.  

\begin{lemma}\label{L:face}
If~$\lambda\leq\frac{1}{k}$, then~$F=\bigcup_{i=1}^kF_i$. 
\end{lemma}
\begin{proof}
Let~$f=a_1v_1+\cdots+a_kv_k\in F$, where~$a_1+\cdots+a_k=1$. Since~$\lambda\leq\frac{1}{k}$, there is an~$l$ such that~$a_l\geq\lambda$. 
Let~$f'=a_1'v_1+\cdots+a_k'v_k$, where
\begin{equation*}
a_i'=\begin{cases}
a_l-\lambda, & \textrm{~if~} i=l; \\
a_i, & \textrm{~if~} i\neq l. 
\end{cases}
\end{equation*}
Then we have~$a_i'\geq 0$ and~$\sum_{i=1}^k a_i'\leq 1$, 
so~$f'\in S$. Hence~$f=\lambda v_l+f'\in \lambda v_l+S=S_l$, and~$f\in F_l$. 
\end{proof}

\begin{lemma}\label{L:simplex}
Let~$S'$ be the simplex generated by~$(1+\lambda)v_1,\ldots,(1+\lambda)v_k$, that is,
\begin{equation*}
S':=\{a_1(1+\lambda)v_1+\cdots+a_k(1+\lambda)v_k\mid a_1+\cdots+a_k\leq 1,a_1,\ldots,a_k\geq 0\}. 
\end{equation*}
If~$\lambda\leq\frac{1}{k}$, then~$S'=S\cup(S_1\cup\cdots\cup S_k)$. 
\end{lemma}
\begin{proof}
Suppose~$s=a_1(1+\lambda)v_1+\cdots+a_k(1+\lambda)v_k\in S'$, where~$a_1+\cdots+a_k=1$. 
Let~$b_i=a_i(1+\lambda)$. If~$s\not\in S$, we have~$1<b:=b_1+\cdots+b_n\leq 1+\lambda$.
Let~$c_i=\frac{b_i}{b}$, and~$t=c_1v_1+\cdots+c_kv_k$. Then~$c_1+\cdots+c_k=1$, so~$t\in F$. 
By Lemma~\ref{L:face}, there is an~$l$ such that~$t\in F_l$. This guarantees~$b_l>c_l\geq\lambda$, and all coefficients of 
\begin{equation}\label{eq1}
s-\lambda v_k=b_1v_1+\cdots+b_nv_n-\lambda v_k
\end{equation}
are non-negative. At the same time, the sum of the coefficients of~(\ref{eq1}) satisfy that~$b_1+\cdots+b_n-\lambda\leq (1+\lambda)-\lambda=1$. Hence~$s\in \lambda v_k+S=S_k$.  
\end{proof}

Now we are ready to prove Theorem~\ref{T2}.

\begin{proof}[Proof of Theorem \ref{T2}]
For~$i\geq 0$, let~$U_i$ be the simplex generated by $$(k+i)v_1,\ldots,(k+i)v_k,$$ and~$V_i=U_i\cap\mathbb{N}^k$. 
Note that~$V=\bigcup_{i=0}^\infty V_i$, so we need to check~$V_i\subseteq\mathrm{FS}(X)$ for every~$i$. 
We will prove it by induction on~$i$. The case of~$i=0$ is trivial. Now we assume that~$V_i\subseteq\mathrm{FS}(X)$. 
Let~$U_{ij}=v_j+U_i$. Then~$\lambda=\frac{1}{k+i}\leq\frac{1}{k}$, and Lemma~\ref{L:simplex} implies
$$
U_{i+1}=U_i\cup(U_{i1}\cup\cdots\cup U_{ik}).
$$
For any~$v\in (V_{i+1}\setminus V_i)\subseteq (U_{i+1}\setminus U_i)$, there is an~$l$ and a~$v'\in U_i$ such that~$v=v'+v_l$. 
Since both~$v$ and~$v_l$ are integer vector,~$v'$ is also an integer vector, so~$v'\in V_i$. 
By the inductive hypothesis,~$v'$ can be expressed as a finite sum of distinct elements in~$X$. 
Add~$v_l$ to it, we can get a finite sum expression of~$v$. There is only one problem need to be settled down:~$v_l$ may occurs~$2$ times in the expression. 
Note that there are finite many of elements of~$X_l$ appear in the expression of~$v'$. 
So we can find an~$m$ such that~$v_l+2v_l+\cdots+2^mv_l$ is in the expression of~$v'$, but~$2^{m+1}v_l$ is not in it. 
Then instead of add~$v_l$ to it, we can replace~$v_l+2v_l+\cdots+2^mv_l$ by~$2^{m+1}v_l$ to get an expression of~$v$. 
Then there is no repeated elements in it. Hence~$V_{i+1}\subseteq\mathrm{FS}(X)$. 

Finally, notice that~$|X(n)|\leq |S|+\sum_{i=1}^k|X_i(n)|\leq c+k\log_2n$, and~$|V(n)|\sim c'n^k$, 
for some constant~$c,c'$ and as~$n\rightarrow\infty$.
Hence~$X$ is thin and we finished the proof. 

\end{proof}

\vfill

\section{On structure of reducible Subset sums of splittable sets  in $\N^2$}

Recall a sequence $A$ is splittable if any element $a_i$ can be split into sum of earlier elements, if $i>i_0$. The simplest example for splittable set is the set of powers of 2.

In this section, then, we consider the subset sum structure of $X$ where
$X:=\{2^m\}_{m =0}^\infty \times \{2^k\}_{k=0}^\infty.$ Our argument can also be applied to other divisible sequences too. This study is motivated by a result from [BHP]. 

We are going to show that there is a structured set which is contained in $FS(X)$, but surprisingly there is also an exceptional set $E$. We will show that the exceptional set contains an arbitrary large 'empty' part (Proposition \ref{empty squares in the exceptional part}), and sometimes a 'dense' part as well (Proposition \ref{dense squares in the exceptional part}).

\begin{prop}\label{E is the interesting part}
Let $E:=\{(a,b)\in \N^2: \ b\leq \log_2 a\}\cup \{(a,b)\in \N^2: \ a\leq \log_2 b\}$. Then $(\N^2\setminus E)\subseteq FS(X)$. 
\end{prop}

\begin{proof}
Since the set $X$ is reflected to the line $y=x$, it is enough to show that every $(a,b)\in \N^2\setminus E$ with $b\leq a$ is an element of $FS(X)$. Denote by $n$ and $m$ the number of terms in the dyadic expansion of $a$ and $b$ respectively. First assume that $n \leq m$. Let $a=2^{c_1}+...+2^{c_n}$ be the dyadic expansion of $a$ and $b=2^{d_1} +...+2^{d_m}$ of $b$. Let us note that if $c_i >0$, then we can write 
$$ a=2^{c_1}+...+2^{c_i-1}+2^{c_i-1}+...+2^{c_n},$$
which shows that $a$ can be written as the sum of $n+1$ powers of two. Repeating this process, we can write $a$ as the sum of $n+1,n+2,...,a$ many sums of powers of two (since $a\geq b$ this process is indeed terminates). Since $n \leq m$ we have that $n \leq m \leq a$ and hence we can write $a$ as the sum of powers of two such that we use exactly $m$ many terms (repetitions are allowed). We simply match this terms arbitrarily with the terms that appear in the dyadic expansion of $b$ and summing this pairs will witness that $(a,b) \in FS(X).$

Assuming  $n > m$ we have that $m < n \leq \log_2 a < b$, so we can do the same trick as above, that is, we start with dyadic expansion of $b$ and split a term in to sum of two two powers and we split another term into the sum of two two powers. At some point, $b$ is written as
is the sum of two powers using exactly $n$ members. Then matching these two powers with the two powers that appear in the dyadic expansion of $a$ we are done.
\end{proof}

In the next proposition we show that the above $E$ contains arbitrarily large empty squares.

\begin{prop}\label{empty squares in the exceptional part}

Let $E$ be the same set as in Proposition \ref{E is the interesting part}. For every $D\in \N$ there exists a square $S_D:=\{(s_1,s_2): \ x_0\leq s_1\leq x_0+D; y_0\leq s_2\leq y_0+D \}\subseteq E$ (for some $x_0,y_0\in \N$) such that $FS(X)\cap S_D=\emptyset$.
\end{prop}

\begin{proof}
First let us remark that the shortest binary representation of a natural number - allowing repetition as well - is the dyadic expansion (indeed if there are two equal terms, say $2^i$ occurs at least twice, then one can replace it by $2^{i+1}$ shortening the number of terms, and so on).

Let $D\in \N$, and for the square $S_D$ we define the lower left corner of $S_D$ $(x_0,y_0)$ as follows: let $x_0:=\sum_{i=D+1}^{2D+1}2^i$ and $y_0:=1$. We show that $(x_0+j,1+k) \notin FS(X)$ for every $1 \leq  j,k \leq D$, so $S_D \subseteq E$. Suppose the contrary that $(x_0+j,1+k) \in FS(X)$ for some $1 \leq j,k \leq D$ and write it as a sum of distinct terms from $X$:
$$
(x_0+j,1+k)= (2^{n_1}, 2^{m_1}) +...+(2^{n_l},2^{m_l}).
$$

Clearly $l \leq 1+k \leq D+1$, on the other and recall that the shortest binary representation is the dyadic expansion, hence $l \geq D+2$, a contradiction.
\end{proof}

Nevertheless the set $E$ is not "empty". It contains "many" lattice points from $FS(X)$:

\begin{prop}\label{dense squares in the exceptional part}
For every $M\in \N$, $M>M_0$, there exists a square $S_M:=\{(t_1,t_2): \ z_0\leq t_1\leq z_0+M; w_0\leq t_2\leq w_0+M \ \}\subseteq E$ (for some $z_0,w_0\in \N$) such that $$|FS(X)\cap S_M|\geq \frac{1}{4}M\log_2M.$$
 $M_0$ can be chosen to $2^{64}$.
\end{prop}

\begin{proof}
Let $R$ be defined as $2^R\leq M< 2^{R+1}$. Let
$$
S_M:=\{(t_1,t_2)\in \N^2: \ 2^{2^{R+1}}\leq t_1\leq 2^{2^{R+1}}+M; 0\leq t_2 \leq M  \}
$$
It is easy to see that $S_M \subset E$. Consider the following sub-rectangle of $S_M$, (since this rectangle will have still "many" elements from $FS(X)$):
$$
S_{2^R-1,2^R}:=\{(t_1,t_2)\in \N^2: \ 2^{2^{R+1}}\leq t_1\leq 2^{2^{R+1}}+2^R-1; 0\leq t_2\leq 2^R \}
$$
Consider any element from $FS(X)$ which is a sum of 'horizontal' elements, i.e. $(n,m)=(2^{i_1},2^f)+(2^{i_2},2^f)+\dots +(2^{i_k},2^f)$,  denote this set by $Z$, more formally $Z:=\{(n,k2^f): f \in \mathbb{N} \}$. We will show that $Z \cap S_{2^R-1,2^R}$ still has a lot of elements. By the choice of $S_{2^R-1}$ we have that for  every element $k \in S_{2^R-1,2^R}$: $k \in \{1, \dots R+1 \}$. Furthermore, for a fixed $k \in \{1, \dots R+1 \}$ ($k$ is the number of terms in the representation):
$$
| \{ n:  2^{2^{R+1}}\leq n \leq 2^{2^{R+1}}+2^R-1  \}| = \binom{R}{k-1}. 
$$

Thus if  $(n,m)=(n,k2^f) \in Z\cap S_{2^R-1,2^R}$ we need that $k2^f \leq 2^R$. So $f$ should be the element of the following set: $\{0,1, \dots , \lfloor R-\log_2 k \rfloor \} $. 

Hence $|Z\cap S_{2^R-1,2^R}|$ is at least: 
$$\sum\limits_{k=1}^{R+1} \binom{R}{k-1} (\lfloor R-\log_2 k \rfloor +1 )  \geq \sum\limits_{k=0}^{R} \binom{R}{k} ( R-\log_2 (k+1)  ).$$

Now for $R \geq 6$, we have $\frac{R}{2} \geq \log_2(R+1)$, so

$$\sum\limits_{k=0}^{R} \binom{R}{k} ( R-\log_2 (R+1) ) \geq \sum\limits_{k=0}^{R} \binom{R}{k} \frac{R}{2}= 2^R \frac{2R}{4} \geq \frac{M}{2} \frac{\log_2 M}{2}.$$

\end{proof}

\section{On the set $FS(\{a_m\}\times\{b_k\})$ where $\{a_m\}$ is denser}

In the previous section we saw that the set $E$ contains "many" grid points from $FS(X)$. One can ask what happens if we change one of the factors of the set $X=\{2^m\}\times\{2^k\}_{m,k\in \N}$ to a denser one.

Let $A=\{a_1<a_2<\dots\}$ be a set of positive integers, where $A(n)>n^{1/2}f(n)$, where $f(n)$ tends to infinite as $n$ tends to infinity. 

Furthermore let $B=\{b_1<b_2<\dots\}$ be an infinite sequence of positive integers.
Let $Y=\{a_m\}_{m\in \N}\times\{b_k\}_{k\in \N}$. We will prove that $FS(Y)$ contains well-structured subsets.
\begin{defi}[Finite generalized arithmetic progression]
A finite generalized arithmetic progression or briefly a GAP of dimension $D$ is the set of the form
$$
\{x_0+l_1d_1+l_2d_2+\dots +l_Dd_D: 0<l_i\leq L_i; \ i=1,2,\dots,D\},
$$
where $x_i, l_i,d_i\in \N$ ($i=1,2,\dots,D)$. The GAP is homogeneous if $x_0=0$. The {\it size} of the GAP is $L_1L_2\cdots L_D$, and GAP is proper if the cardinality of the set above is equal of its size.
\end{defi}

\begin{thm}\label{4.1}
Let $T>T_0$ be an integer. 

1. For every $(L_1,L_2,\dots ,L_D)$ $D-$tuples there exists a proper and homogeneous GAP with dimension $D$ and size $L_1L_2\cdots L_D$ contained in $FS(Y)$.

2. There are infinitely many rectangular $R$ such that $|FS(Y)\cap R|\gg|R|\frac{B(T)}{T}$. In particular if $\limsup_{T\to \infty}\frac{B(T)}{T}>0$ then $|FS(Y)\cap R|\gg|R|$ holds infinitely many times.
\end{thm}
\begin{proof}[Proof of Theorem \ref{4.1}]

First we prove 1. 

In the first step we prove the existence of an arithmetic progression of length $L_1$. This will be $GAP(1)$. Then successively if $GAP(i)$ is defined then we show the existence of $GAP(i+1)\supseteq GAP(i)$. For the $i^{th}$ step $GAP(i)$ to be proper, it is sufficient to prove that $d_i>\sum^{i-1}_{j=1}L_jd_j$.
Consider the disjoint union of the sequence $A=A_1\sqcup A_2\sqcup \dots \sqcup A_D$, where $A_i(n)>\sqrt{n}f(n)/D$ holds for $i=1,2,\dots, D$. $GAP(i)$ will be constructed on $\sqcup_{j=1}^i A_j$.

We now construct $GAP(1)$ in $A_1$. Let $n$ be an integer for which $f^2(n)/8D^2>L_1$. Write shortly $L=L_1$. $A_1(n)>\sqrt{n}f(n)/D$ and let $A_1\cap [1,n]:=A'_n=\{a_1<a_2<\dots a_v\}\subseteq [1,n]$. Denote $r(x)$ the sum-representation of $A'_n$, $r(x)=\{(a,a'): a+a'=x; \ a< a'; \ a,a'\in A'_n\}$. Now $A'_n+A'_n\subseteq [2,2n]$ and  
$$
nf^2(n)/4D^2\leq |A'_n|^2/4\leq \frac{|A'_n|(|A'_n|-1)}{2}=\sum_xr(x)\leq 
$$
$$
\leq \max_xr(x)|A'_n+A'_n|\leq \max_xr(x)(2n)
$$
and so $r(x_1):=\max_xr(x)>f^2(n)/8D^2$. Hence we have $x_1=a_{j_1}+a_{j_2}=a_{j_3}+a_{j_4}=\dots =a_{j_{{2m-1}}}+a_{j_{2m}}$, where  $m>f^2(n)/8D^2>L$. 

Let $b_1,b_2\in B$ and consider the points from $FS(Y)$: $v=(a_{j_1}+a_{j_2},b_1+b_2)= (a_{j_3}+a_{j_4},b_1+b_2)=\dots =(a_{j_{{2L-1}}}+a_{j_{2L}},b_1+b_2)$. Since $m>L$ and the elements $a_{j_1},a_{j_2},\dots, a_{j_{{2L-1}}},a_{j_{2L}}$ are pairwise different we have that $2v=(a_{j_1}+a_{j_2}+a_{j_3}+a_{j_4},2b_1+2b_2),\dots$, $Lv=(a_{j_1}+a_{j_2}+a_{j_3}+a_{j_4}+\dots +a_{j_{{2L-1}}},a_{j_{2L}},Lb_1+Lb_2)$ are elements of $GAP(1)\subset FS(Y)$  which proves the first step. Note that $d_1=x_1$.

Assume that the $i^{th}$ $GAP(i)$, ($1\leq i< D$), $GAP(i)\supseteq GAP(i-1)\supseteq\dots \supseteq GAP(1)$ has been defined on $\sqcup_{j=1}^i A_j$ and the differences $d_1,d_2,\dots, d_i$ have been determined. We are going to proof of the existence $GAP(i+1)$ on $\sqcup_{j=1}^{i+1} A_j$. First we use the sequence $A_{i+1}$. Actually we follow the previous argument. Our requirement now is that let $f^2(n)/16D^2>\max \{\sum^i_{j=1}d_jL_j, L_{i+1}\}$. As before we get an integer $x_{i+1}$ for which the number of representation of $x_{i+1}$ in the form $a+a'$, $a<a'$ and $a,a'\in A_{i+1}$ is at least $f^2(n)/8D^2>2\max \{\sum^i_{j=1}d_jL_j, L_{i+1}\}$. Since for all integer $x$, $x>2r(x)$ we get that $d_{i+1}:=x_{i+1}>\max \{\sum^i_{j=1}d_jL_j, L_{i+1}\}$ and we can select an arithmetic progression of length $L_{i+1}$. Denote this arithmetic progression by $AP(i+1)$. Finally note that $GAP(i+1)=GAP(i)+AP(i+1)$. Since for every $i\geq 1$ $d_{i+1}>\sum^i_{j=1}d_jL_j$, we can avoid having two different representations in the $GAP(i+1)$, which proves the first part of the theorem.

Note that in the present proof we used only that the set $B$ is non-empty.
\medskip

Now we prove 2. 
For the proof, we need the following deep result of Szemer\'edi and Vu, which Chen also proved independently. (see [SZV] and [C]):
\begin{lemma}\label{4.2}
There is a positive constant $c$ such that the following holds. Any increasing sequence $A=\{a_1<a_2<\dots\}$ with $A(n)>c\sqrt{n}$. We have that $FS(A)$ contains an infinite arithmetic progression.
\end{lemma}
Split $A$ into two parts $A=A_1\bigsqcup A_2$, where $A_1(n)>c\sqrt{n}$  for $n>n_0$, and $A_2=A\setminus A_1$. By Lemma \ref{4.2}, $FS(A_1)$ contains an infinite arithmetic progression, $AP(d)$ with difference $d$. Choose the interval $I=[1,H]$ such that the relative density of this arithmetic progression in $I$ is at least $1/2d$, i.e. $|I\cap AP(d)|>|I|/2d$.  For any $x\in I\cap AP(d)$, denote by $trm(x)$ the maximum number of members taken from $A_1$ in the representation $FS(A_1)$; (formally
$$
trm(x)=\max\Big\{\sum\varepsilon_i\mid x=\sum\varepsilon_ia_i; \ a_i\in A_1, \ \varepsilon_i \in \{0,1\}\Big\}.
$$
Let $Q= \max \{trm(x): x\in I\cap AP(d)\}$. Choose $A_Q:=\{a_1,a_2, \dots ,a_Q\}$ from $A_2$. Since $A_2(n)>\sqrt{n}(f(n)-c)$ this $Q-$tuple exists. Let $A'_1=A_1\cup A_Q$. Clearly $I':=\sum_{1\leq j\leq Q}a_j+(I\cap AP(d))\subset FS(A'_1)$ and for every element $y\in I'$, $trm(y)\geq Q+1$.

Select an element $x$ from $I'$. Write $Q_x=trm(x)$. 
Let $B_T=B\cap [1,T]$. Take $Q_xB_T:=B_T+B_T+\dots +B_T$ ($Q_x$ many times). It is well-known that for every finite set $|A+A|\geq 2|A|-1$ holds and so by induction we have $|Q_xB_T|\geq Q_x|B_T|-(Q_x-1)\gg  Q_x|B_T|$, and clearly $Q_xB_T\subseteq [1, Q_xT]\subseteq [1, 2QT]$, since every $x\in I'$, $Q+1\leq trm(x)\leq 2Q$.

Now we are in the position to  define the rectangle $R$. Let the horizontal segment of $R$ be $I':=[a,b]$ and the length of the vertical segment be $2QT$, i.e. $R=\{[a,b]\times [1, 2QT]\}\cap \N^2$. The number of lattice points in $R$ is $(1+o(1))2|I'|QT$. While 
$$
|FS(Y)\cap R|\geq |FS(A'\times B)\cap R|\gg\sum_{x\in I'}Q_x|B_T| \gg B(T)Q|I'|=
$$
$$
=\frac{B(T)|I'|QT}{T}\gg|R|\frac{B(T)}{T}.
$$

\end{proof}

It has not escaped our attention that in the second part of the theorem, that it would have been enough to assume that $A(n)\geq c\sqrt{n}+B(n)$, where $c$ is optimal  (in the sense that for every $c'<c$ Lemma \ref{4.2}
 does not hold) and $B(n)$ tends arbitrarily slowly to infinity. We just didn't want to complicate our formulation.

\medskip

We ensured arbitrary long arithmetic progression in $FS(Y)$. The question is, does it contain an infinite arithmetic progression?  

Again the answer depends on the 'slope' of the arithmetic progressions (compare with Proposition 1.1). The following example shows that a "horizontal" arithmetic progression is possible.

\begin{example}
Let $S$ be the sequence of square numbers. It is known that $n$ is expressible as a sum of 5 distinct non-zero squares provided $n\geq 1024$ (see e.g. [B]). Let $Z=S\times \{1\}$. Clearly $\{(n,5):\ n\geq 1024\}\subseteq FS(Z)$.
\end{example}
So in the rest we exclude arithmetic progressions which are 'parallel' to the two axes. Let us call this type of arithmetic progression type B. In this case, we can say no more: 

\begin{prop}
For every sequence of integers $A$, there exists an infinite sequence of integers $B$ such that $FS(A\times B)$ does not contain an infinite arithmetic progression. 
\end{prop}

\begin{proof}
Let us enumerate the set of arithmetic progressions type B (they form a countable set). Write them as $AP_1,AP_2,\dots$. For each $AP_i$ we define an element $b_i\in B$. 
Let $b_1$ any positive integers. We define a function $\alpha_1: \N\mapsto \N$ for which for all $(n,y)\in FS(A\times B)$, $\alpha(n)\geq y$ (note that not only for one $y$ rather for each $(n,y)\in FS(A\times B)$, when $n$ is fixed). Represent $n$ in the way when the $trm(n)$ is maximal. So $n=a_{i_1}+a_{i_2}+\dots +a_{i_t}\geq 1+2+\dots t$, hence $\max trm(n)\leq 2\sqrt{n}$. Hence $\alpha_1(n)=2b_1\sqrt{n}$ is an appropriate function. If $AP_1$ lies on the line $g(x)=y=mx+b$ then there is an $x_0\in \R$, such that for every $x>x_0$, $g(x)>2b_1\sqrt{x}$ . Furthermore, there exists an explicit computable $x_1\in \N$ such that between $(x_0,g(x_0))$ and $(x_1,g(x_1))$ the line contains at least one element from $AP_1$. Call it an exceptional point. Now let $b_2=\min \{m>g(x_1): \ m\in \N\}$.

The further elements of $B$ are defined in a similar way; if $b_i$ has been defined then for all $(n,m)\in FS(A\times \{b_1<\dots <b_i\})$ we have $m\leq (\sum_{1\leq j\leq i-1}b_j)\sqrt{n}$ and then the definition of $b_{i+1}$ is similar as before.

Assume that the sequence $B$ has been defined. Consider say an element $(n,m)\in FS(A\times B)$, $x_0\leq n\leq x_1$. If any $b_i$, $i\geq 2$ is a term of the representation of $m$ then $m\geq b_2$. If  just $b_1$, then $m\leq 2b_1\sqrt{n}$. The given exceptional point in the $AP_1$ is not an element of $FS(A\times B)$.
The further argument is the same for any $AP_i$.
\end{proof}

{\bf Declaration of competing interest}

The authors declare that they have no known competing financial interests or personal relationships that could have
appeared to influence the work reported in this paper.

\smallskip

{\bf Data availability}

No data was used for the research described in the article.

\section*{Acknowledgment}

The first named author is supported by the National Research, Development and Innovation Office NKFIH Grant No K-146387.


\begin{thebibliography}{99}

\bibitem[B]{-1} P.T.Bateman, A.J.Hildebrand, and G.B.Purdy: Sums of distinct squares, Acta Arith. 67:4 (1994), pp. 349–380. 


\bibitem[BHP]{0} Bakos B., Hegyvári N., Pálfy M.
On a communication complexity problem in combinatorial number theory
Moscow Journal of Combinatorics and Number Theory  10 : 4 pp. 297-302. , 6 p. (2022)

\bibitem[C]{0b} Yong-Gao Chen, On subset sums of a fixed set Acta Arithmetica 106 (2003), 207-211
DOI: 10.4064/aa106-3-1

\bibitem[GFH]{1} Chen, YG ; Fang, JH ; Hegyv\'ari, N
Erd\H os-Birch type question in $\N^r$
J. of Number Theory 187 pp. 233-249. , 17 p. (2018)

\bibitem[H1]{2} Hegyv\'ari, N. Complete sequences in $\N^2$ European Journal of Combinatorics (17) (1996)  741-749.


\bibitem[H2]{4} Hegyv\'ari, Thin complete subsequence. 
Ann. Univ. Sci. Budap. Rolando E\"otv\"os, Sect. Math. 44, 151-156 (2001).

\bibitem[N]{5} Nathanson, Melvyn B. Thin bases in additive number theory, Discrete Math. 312, No. 12-13, 2069-2075 (2012)

\bibitem[S]{6} Spencer, Joel (1996). "Four Squares with Few Squares". Number Theory: New York Seminar 1991–1995. Springer US. pp. 295–297. doi:10.1007/978-1-4612-2418-1-22

\bibitem[SZV]{6b} E. Szemer\'edi, V. Vu, Finite and infinite arithmetic progressions in sumsets, Annals of Mathematics, 163 (2006), 1-35, https://doi.org/10.4007/annals.2006.163.1

\bibitem[W]{7} Eduard Wirsing, Thin Subbases, Analysis 6 (1986), 285–30

\end{thebibliography}
\end{document}